\documentclass[10pt]{article}
\usepackage{amsmath,amsfonts,amsthm,epsfig}
\usepackage{color}
\usepackage[english,ngerman]{babel}
\usepackage[T1]{fontenc}
\usepackage[ansinew]{inputenc}
\newcommand{\norm}[1]{\left\| #1 \right\|}

\newcommand{\abs}[1]{\left| #1 \right|}
\newcommand{\RR} {\mathbb R}
\newcommand{\NN} {\mathbb N}

\newcommand{\supp} {\operatorname{supp}}
\newcommand{\HM} {{\cal H}}
\newcommand{\HMS} {\omega_{N-1}}
\newcommand{\LM} {{\cal L}}

\newcommand{\dist}[2] {\operatorname{dist}\left(#1;#2\right)}

\newcommand{\loc} {\text{loc}}
\newcommand{\eps} {\varepsilon}

\newcommand{\lam} {\lambda}

\newcommand{\mbf}[1]{\boldsymbol{#1}}
\newcommand{\Nemytskii}{Nemytskii}

\newcommand{\Poincare}{Poincar\'{e}}
\newcommand{\Hoelder}{H\"{o}lder}

\newcommand{\EL}{Euler--Lagrange}
\newcommand{\DuBoisR}{Du Bois--Reymond}
\newcommand{\Lebesgue}{Lebesgue}
\newtheorem{thm}{Theorem}[section]
\newtheorem{cor}[thm]{Corollary}
\newtheorem{lem}[thm]{Lemma}
\newtheorem{prop}[thm]{Proposition}
\theoremstyle{definition}
\newtheorem{defn}[thm]{Definition}

\theoremstyle{remark}
\newtheorem{rem}[thm]{Remark}
\author{Stefan Kr\"{o}mer\\
Institut f\"{u}r Mathematik \\ Lehrstuhl
f\"{u}r Nichtlineare Analysis \\ Universit\"{a}t Augsburg \\
Email: stefan.kroemer@math.uni-augsburg.de}
\date{}
\title{Existence and symmetry of minimizers for nonconvex radially symmetric variational problems}
\begin{document}
%%%%%%%%%%%%%%%%%%%%%%%%%%%%%%%%%%%%%%%%%%%%%%%%%%%%%%%%%%%%%%%%%%%%%%%%%%%%%%%
%%%%%%%%%%%%%%%%%%%%%%%%%%%%%%%%%%%%%%%%%%%%%%%%%%%%%%%%%%%%%%%%%%%%%%%%%%%%%%%
%%%%%%%%%%%%%%%%%%%%%%%%%%%%%%%%%%%%%%%%%%%%%%%%%%%%%%%%%%%%%%%%%%%%%%%%%%%%%%%
\selectlanguage{english}
\maketitle{}
%%%%%%%%%%%%%%%%%%%%%%%%%%%%%%%%%%%%%%%%%%%%%%%%%%%%%%%%%%%%%%%%%%%
\begin{abstract}
%\noident
We study functionals of the form
%
%\vspace{-2ex}
\begin{equation*}
    E(u):=\int_{B_R(0)} W(\nabla u)+G(u)\,dx,
\end{equation*}
%\vspace{-2ex}
%
where $u$ is a real valued function over the ball $B_R(0)\subset \RR^N$ which vanishes on the boundary 
and $W$ is nonconvex. The functional is assumed to be radially symmetric in the sense that 
$W$ only depends on $\abs{\nabla u}$. 
Existence of one and radial symmetry of all global minimizers is shown with an approach based on convex relaxation. Our assumptions on $G$ do not include convexity, thus extending a result of A.~Cellina and S.~Perrotta.
% \cite{CePe94a}. 
\end{abstract}
%\newpage
%%%%%%%%%%%%%%%%%%%%%%%%%%%%%%%%%%%%%%%%%%%%%%%%%%%%%%%%%%%%%%%%%%%%%%%%%%%%%%%
%%%%%%%%%%%%%%%%%%%%%%%%%%%%%%%%%%%%%%%%%%%%%%%%%%%%%%%%%%%%%%%%%%%%%%%%%%%%%%%
%%%%%%%%%%%%%%%%%%%%%%%%%%%%%%%%%%%%%%%%%%%%%%%%%%%%%%%%%%%%%%%%%%%%%%%%%%%%%%%
\section{Introduction}

%\markboth{}{}
%\thispagestyle{myheadings}
%%%%%%%%%%%%%%%%%%%%%%%%%%%%%%%%%%%%%%%%%%%%%%%%%%%%%%%%%%%%%%%%%%%%%%%%%%%%%%%
%%%%%%%%%%%%%%%%%%%%%%%%%%%%%%%%%%%%%%%%%%%%%%%%%%%%%%%%%%%%%%%%%%%%%%%%%%%%%%%
%%%%%%%%%%%%%%%%%%%%%%%%%%%%%%%%%%%%%%%%%%%%%%%%%%%%%%%%%%%%%%%%%%%%%%%%%%%%%%%
This paper is concerned with the variational problem
arising from the energy functional
\begin{equation*}\label{E0}\tag{\ensuremath{E}}
    E(u):=\int_{B_R(0)} \left[W(\nabla u)+G(u)\right]\,dx,
\end{equation*}
where $u$ is a scalar field on
$B_R(0)=\{x\in \RR^N\mid \abs{x}<R\}\subset \RR^N$ ($N\geq 2$) which vanishes on the boundary. 
Simple examples for the functions $W$ and $G$ considered
are $W(\xi):=(\abs{\xi}^2-1)^2$ and $G(\mu):=-\mu^2$. The primary
qualitative features of $W$ are that it is continuous, nonconvex, coercive and
radially symmetric in the sense that it only depends on the euclidean
norm of its argument. It may have more wells 
than the two in the example above, however, and convexity of $W$ at infinity (i.e., if the norm of its argument is large enough) is not assumed. 
Besides the prototype above, our assumptions on $G$ in particular include all functions of class $C^2$ which are strictly monotone and do not grow too fast. Moreover, the monotonicity assumptions on $G$ can be dropped if $0\in \RR^N$ is the unique minimizer of $W$.
%In particular, convexity of $G$ is not assumed. 
%Both minima and
%(more general) critical points are discussed, asking the questions
%of existence, uniqueness and qualitative properties.

Abundant literature addressing the existence and further properties
of global minimizers of nonconvex variational problems is available. 
For an overview of known results in the case on nonconvex simple integrals ($N=1$),
we refer to \cite{Ray87b,CePe02a} and the references therein. 
In higher dimensions ($N>1$), conditions for attainment have been obtained even without assuming
symmetry (in  particular, the domain does not have to be a ball, then), see for example \cite{Ray87a,Ce99a,CeCuGui04a}. 
Generalizations for vector--valued $u$ are obtained in \cite{CeCo90,Ray94a,CaTah99a} ($N=1$) and
\cite{Ray92a} ($N>1$). 
For the most general existence result for autonomous functionals 
%(i.e., with an integrand which does not explicitly depend on $x$) 
and further references, the reader is referred to \cite{CeCuGui04a}. 
%There, integrands of the general form $W(\nabla u,u)$ are treated (as opposed to the more special additive form %considered here). 
In the case of our energy $E$, the existence of a minimizer of $E$ in $W_0^{1,p}$ follows from the results in \cite{CeCuGui04a} if $G$ does not have strict local minima and (roughly speaking) does not oscillate too fast, provided that $W$ satisfies (\ref{Wcaffine}) below. 
Still, some open questions remain. In particular,
to ensure existence of a minimizer, all of the above mentioned
papers for $N>1$ have to assume that the convex envelope $W^{**}$ of $W$ has the following property:
\begin{align}\label{Wcaffine}
    W^{**}~\text{is affine on any component of the detachment
    set}~\{W^{**}<W\}.
\end{align}
However, this behavior of the convex envelope is by no means
typical. Usually, $W^{**}$ will be affine only along suitable
one--dimensional lines wherever it differs from $W$.
Our radially symmetric prototype example above of course satisfies
(\ref{Wcaffine}), but no multi-well potential $W$ whose set of global minima
consists of a finite number of points has this property, and even if $W$
is radially symmetric, any nonconvex parts outside the outermost sphere of
minima are ruled out. If $G=0$, \eqref{Wcaffine} is known to be necessary for attainment
for \emph{arbitrary} Dirichlet boundary conditions \cite{Ce93a,Frie94a}. If $G$ is strictly concave and/or strictly monotone, examples are rare. For instance, even strictly concave $G$ cannot always guarantee existence as it would in the one--dimensional case (treated in \cite{CeCo90}) if \eqref{Wcaffine} fails to hold, see \cite{diss05B} (Section 1.4). 

The radially symmetric case is studied in
\cite{CePe94a,Cra00a,Ga01a,CraMu05a}. 
There, (\ref{Wcaffine}) can be dropped provided that $G$ is convex and decreasing, a result first stated in~\cite{CePe94a} (see also \cite{CePe05}, where an error in the proof of Theorem 2 in \cite{CePe94a} is corrected). 
A generalization for vector--valued $u$ can be found in
in \cite{CraMu01a}.
%The convexity of $G$ is also exploited to
%prove the existence of a radially symmetric minimizer. 
Here, we show in particular that convexity of $G$ is actually
a technical assumption in the sense that it can be dropped if $G$ is
of class $C^2$. Our proof of existence follows a path which is somewhat standard for 
nonconvex variational problems: 
First, we study the relaxed functional $E^{**}$, where $W$ is replaced by its convex envelope $W^{**}$ 
and show that $E^{**}$ has a radially symmetric minimizer $u$. In a second step, we prove that $u$ satisfies $W(\nabla u)=W^{**}(\nabla u)$ a.e.~by extending the ideas developed in 
\cite{CePe94a,CePe05}. As a consequence, $u$ also is a minimizer of the original problem.
Let us emphasize that this second step is by no means trivial. 
Of course, if the functional is restricted to the class
of radially symmetric functions, it can be rewritten as a single integral 
which in our case leads to
\begin{align}\label{E1d}
	\tilde{E}(u):=\int_0^R r^{N-1}\tilde{W}(u')+r^{N-1}G(u)\,dr,
\end{align}
where $\tilde{W}(\pm\abs{\cdot})=W(\cdot)$. 
Still, the available results in the one--dimensional case 
cannot be applied. This is inhibited by the lack of a boundary condition at $r=0$ and the 
singular weight $r^{N-1}$. Even worse, if one 
is willing to ignore the aforementioned problems for the time being, the 
main conditions on the integrand entailing attainment (the one of Theorem 1.2 in \cite{Ray94a} or (C2**) in \cite{CaTah99a}, e.g.) 
fail to hold in general under our assumptions on $W$ and $G$. Here, the main problem arises from the explicit 
dependence on $r$ of the term containing the derivative $u'$ in \eqref{E1d}, despite its simple form. \\
%For instance, let us assume for simplicity that $G'<0$ on $\RR$. 
%Then, 
%to argue as in \cite{Ray87b}, we would have to show that
%$G'(\mu)-\frac{N-1}{r}(\tilde{W}^{**})'(t)\neq 0$ for every 
%$r \in (0,R)$, $\mu>0$ and $t<0$ such that
%$\tilde{W}(t)>\tilde{W}^{**}(t)$, cf.~Theorem 1.2 in \cite{Ray94a}.
%(Here, we consider only positive $\mu$ and negative $t$ because every radially symmetric minimizer $u$ of $E^{**}$ %satisfies $u>0$ and $u'<0$ on $(0,R)$ if $G'<0$ due to Theorem~\ref{1thmradmincon} below.)
%This is true if $\tilde{W}^{**}$ is constant on the detachment set $\{\tilde{W}>\tilde{W}^{**}\}$, 
%which in terms of $W$ means that \eqref{Wcaffine} has to be satisfied. 
%Further technical problems inhibiting the direct use of a one--dimensional result 
%are the lack of a boundary condition at $r=0$ and the 
%singularity caused by the factor $r^{N-1}$ at $r=0$.\\
A related problem on the annulus $a<r=\abs{x}<b$ is studied in \cite{Tah90a}. There, Dirichlet boundary conditions are prescribed which require that $u(a)<u(b)$, where at the same time $G'<0$ on $\RR$ is assumed. Note however that the latter implies that radially symmetric minimizers are strictly decreasing in radial direction if the inner boundary value is free as in our case. 
%Another drawback of the approach of 
%\cite{Tah90a} is that $G$ has to be sufficiently small (i.e, scaled with a sufficiently small factor) and the proof %does not provide an explicit bound. 

The question of symmetry of minimizers, or symmetry of so--called ground states
(positive solutions of variational problems having the least energy among all critical points)
has also received considerable attention in the literature, 
although almost exclusively for problems leading to elliptic equations of second order.
On symmetric domains, symmetry of minimizers or ground states can be obtained using rearrangement techniques (for an overview, see \cite{Ka85B} or \cite{Bae94a}) or reflection arguments. Results in this direction for example can be found in \cite{Lo96a,FluMue98a,Bro04a}. The method of moving planes also has been used with great success \cite{GiNiNi79a,GiNiNi81a,LiNi92b,LiNi93a,SeZou99a,DaPaRa99a}, in particular on unbounded domains with translation invariance which introduces extra difficulties. (Both lists are far from exhaustive.)
In both cases, the proof of symmetry of minimizers (respectively, ground states) is typically based on a maximum principle, to show that a suitable symmetric rearrangement of a minimizer (or a ground state) has to coincide with the original function. Alternatively, one can use characterizations of those functions $u$ 
whose symmetric rearrangement $\hat{u}$ has the same energy as $u$: for example, if $u\in W_0^{1,p}(B_1(0))$ ($p>1$) is nonnegative and $\hat{u}$ denotes its Schwartz symmetrization, then
$\int \abs{\nabla u}^p=\int \abs{\nabla \hat{u}}^p$ implies that either $u=\hat{u}$ or $u$ has a plateau of positive measure below the essential supremum of $u$ (cf.~\cite{BroZie88a}, this is used in \cite{FluMue98a}).
For the purpose of proving symmetry we can assume that $W$ is convex (but not strictly convex!), due to the relaxation theorem (e.g.~\cite{Da89B}, Chapter 5) which implies that every minimizer of $E$ also minimizes the relaxed functional $E^{**}$ where $W$ is replaced by its convex envelope $W^{**}$. 
Still, for both the functionals $E$ and $E^{**}$ considered here, the Euler--Lagrange equation is not elliptic,
since ellipticity, even in a degenerate sense as for example satisfied by the $p$--Laplacian, implies strict convexity of $W$. Hence the use of the maximum principle is out of question. 
If $E^{**}$ has a nonnegative minimizer $u$, then one minimizer is radially symmetric, because the Schwartz symmetrization $\hat{u}$ of $u$ satisfies $E^{**}(\hat{u})\leq E^{**}(u)$ (see for example \cite{Ba80B}).
%; if assumption \eqref{G1} below holds then we may assume that $u\geq 0$). 
Obtaining the symmetry of \emph{every} minimizer is more subtle, though. In particular, it is not difficult to see that the equality
$\int W^{**}(\nabla u)=\int W^{**}(\nabla \hat{u})$ no longer implies that $u=\hat{u}$, if $W^{**}$ is convex but constant on a nonempty open set (even if we assume that $u$ does not have plateaus). 
%The same problem also occurs for various other rearrangements (besides Schwartz symmetrization) which do not change %the measure of super--level sets.
If $G$ is convex and strictly monotone, this difficulty is overcome in \cite{CePe94a},
where a symmetric rearrangement is defined by averaging on concentric spheres.
The disadvantage of this method is that the minimizing property of the rearranged function can only be shown for
convex $G$, using Jensen's inequality.
The main idea in our proof of symmetry is to compare the energy of a given minimizer with the energies of a whole family of radially symmetric functions, obtained from the profiles of the original function 
along all straight lines connecting the center $0$ of $B_R(0)$ to a boundary point (cf.~Lemma~\ref{1lemsymreg}).
This approach also yields symmetry of one minimizer, even without the assumption that a given minimizer is nonnegative. Another advantage lies in the fact that we can also show symmetry of every minimizer provided that $G$ is strictly monotone, using neither strict convexity of $W$ or $W^{**}$
(which, as a byproduct, turns out to be sufficient, too) nor convexity of $G$. Moreover, this technique is purely variational and hence only requires minimal regularity assumptions.% on $W$ and $G$.

Finally we mention that under more restrictive conditions on $W$ and $G$, the global minimizer of $E$ can be obtained
as a singular limit of critical points of a sequence of regularized functionals containing the additional term $\frac{\eps}{2}(\Delta u)^2$ in the integrand, with small $\eps>0$ \cite{KroeKie06a}. In particular, this might provide a good framework for numerical investigations.

The rest of this paper is organized as follows: In the next section, general notation and the main assumptions on $W$ and $G$ are collected. The third section contains results for $E^{**}$, in particular sufficient conditions
for the symmetry of all minimizers, subsumed in Theorem~\ref{1thmradmincon}. They are used in Section~\ref{1secWNCSYM} in the proof of our 
main result, Theorem~\ref{1thmncradmin}, existence of a minimizer and symmetry of all minimizers for nonconvex $W$ 
(and nonconvex $G$).

The results of this paper were presented as a part of the author's PhD thesis \cite{diss05B}. 

\section{Preliminaries\label{1secPre}}
Given two vectors $\xi,\eta\in \RR^N$, $\xi\cdot\eta$ is their
euclidean scalar product. The euclidean norm in $\RR^N$ as well as
the modulus in $\RR$ are denoted by $\abs{\cdot}$, and 
$B_R(a)$ is the open ball in $\RR^N$ 
with radius $R>0$ and center $a\in \RR^N$.
Moreover, $S^{N-1}$ is the boundary of the unit ball in $\RR^N$,
equipped with the $(N-1)$--dimensional Hausdorff measure (if measure--theoretic structure
is needed). The Lebesgue measure and the $s$-dimensional Hausdorff measure of a measurable set $A\subset \RR^N$
are denoted by $\LM_N(A)$ and $\HM_s(A)$, respectively. For the surface area of the sphere, we use the abbreviation $\HMS:=\HM_{N-1}(S^{N-1})$. 
The symbol $\norm{\cdot}$ is used for norms in function spaces, where the
corresponding space will be given in the index, for example
$\norm{\cdot}_{L^p(\Omega)}$. As usual, Sobolev spaces 
of real--valued functions in
$L^p(\Omega)$ which are $k$ times weakly differentiable in
$L^p(\Omega)$ are denoted by $W^{k,p}(\Omega)$, and
$W_0^{k,p}(\Omega)\subset W^{k,p}(\Omega)$ stands for the closure of
the set of infinitely times differentiable functions with compact
support in $\Omega$ (i.e., $C_0^{\infty}(\Omega)$) with respect to
the $W^{k,p}$-norm. The domain $\Omega$ is omitted if it is clear from the context. 
Finally, with a slight abuse of notation, the same letter is used both for a radially symmetric function $u:B_R(0)\to \RR$ and its "profile" $u:(0,R)\to \RR$ related by $u(\abs{x})=u(x)$. In that context,
$u'(\abs{x})=\partial_r u(x):=\nabla u(x)\cdot \frac{x}{\abs{x}}$ denotes the first derivative in radial direction.

Our basic assumptions on $W$ and $G$ are as follows.

Assumptions on $W$:
\begin{alignat*}{2}
    {}& \text{(Regularity)} &&\quad W:\RR^N\to \RR~\text{is continuous,}
    \label{W0}\tag{\ensuremath{W_0}}\\
    {}& \text{(Coercivity)} &&\quad W(\xi)\geq \nu_1\abs{\xi}^{p}-C,
    \label{W1}\tag{\ensuremath{W_1}}\\
    {}& \text{(Growth)} &&\quad \abs{W(\xi)}\leq \nu_2\abs{\xi}^{p}+C,
    \label{W2}\tag{\ensuremath{W_2}}
\intertext{%\end{alignat*}
for every $\xi\in\RR^N$, where $p>1$, $\nu_1\leq \nu_2$ and $C$ are
positive real constants. Furthermore, we assume that $W$ is invariant under rotations:
%\begin{alignat*}{2}
}
    {}& \text{(Symmetry)} &&\quad
    \begin{array}{ll}
        W(\xi)=\tilde{W}(\abs{\xi}),~\text{where}\\
        \text{$\tilde{W}:\RR\to\RR$ is an even function of class $C^0$.}
    \end{array}
    \label{Wsym}\tag{\ensuremath{W_3}}
\end{alignat*}
Note that in particular we do not require $W$ to be convex. If $W$
is nonconvex, the %so-called Maxwell 
points $M$ and $-M$, defined
below, are of special interest:
\begin{equation} \label{WMpoints}
    M:=\max\left\{t\geq 0~\left|~
    \tilde{W}(t)=\min_{s\in\RR}\tilde{W}(s)\right\}\right.\geq 0.
\end{equation}
The case $M=0$ occurs if and only if $0$ is the unique minimizer of $\tilde{W}$.
Another important object in the study of nonconvex $\tilde{W}$ is
its convex envelope (or bipolar)
\begin{align} \label{WWc}
    \tilde{W}^{**}(s):=\sup\left\{V(s)\,\left|\,
    \text{$V:\RR\to\RR$~is convex and $V\leq
    \tilde{W}$}\right.\right\},~s\in\RR.
\end{align}
If $\tilde{W}$ is continuous or of class $C^1$ then the
same holds for $\tilde{W}^{**}$. Furthermore, $\tilde{W}^{**}$ is convex
and affine on any connected component of the set where it differs from
$\tilde{W}$. Also note that $\tilde{W}^{**}$ is constant on $[-M,M]$, and
$\tilde{W}(\pm M)=\tilde{W}^{**}(\pm M)$. 
However, the detachment set $\{\tilde{W}^{**}>\tilde{W}\}$ might contain intervals which are not subsets of $(-M,M)$, in fact even countably many are allowed.
\begin{rem}\label{remskewW}
If $W$ is replaced by a function $\hat{W}$ of the form
$\hat{W}(\xi)=W(\xi)+a\cdot \xi$, where $a\in \RR^N$ is an arbitrary
fixed vector, then the energy $E$ remains unchanged, by virtue
of the Gauss Theorem. In particular, all critical points
persist. This invariance can be used to treat some
cases when $W$ is ``skew'', as opposed to our assumption \eqref{Wsym}.
\end{rem}
Assumptions on $G$:
%}
\begin{alignat*}{2}
    {}& (Regularity) &&\quad \begin{array}{l}G:\RR\to\RR~\text{is continuous,}
    \end{array}
    \label{G0}\tag{\ensuremath{G_0}}\\
    {}& (Growth) &&\quad 
    \begin{array}{ll}
    G(\mu)\geq -\nu_3 \abs{\mu}^{p-\varrho}-C,&\\
    G(\mu)\leq \phantom{-}
    \nu_4 \abs{\mu}^{p^*-\varrho}+C&~\text{if $p<N$},~\text{and}\\
    G(\mu)\leq \phantom{-}
    \nu_4 \abs{\mu}^{\tilde{p}}+C&~\text{if $p=N$, for a $\tilde{p}<\infty$,}
    \end{array}
%    \begin{array}
%    G(\mu)\geq -\nu_3 \abs{\mu}^{p-\varrho}-C,%\\
%    G(\mu)\leq \phantom{-}
%    \nu_4 \abs{\mu}^{p^*-\varrho}+C~\text{if $p<N$},~\text{and} \\
%    G(\mu)\leq \phantom{-}
%    \nu_4 \abs{\mu}^{\tilde{p}}+C~\text{if $p=N$, for an $\tilde{p}<\infty$,}
%    \end{array}
    \label{G2}\tag{\ensuremath{G_1}}
\intertext{for every $\mu\in \RR$, where $C,\nu_3,\nu_4\geq
0$ and $\varrho\in (0,p]$ are constants and $p^*:=\frac{pN}{N-p}$ is the critical Sobolev exponent.
If $M>0$, we also need (partial) monotonicity of $G$:}
    {}& (Shape) &&\quad \begin{array}{l}
    G~\text{is decreasing on $[0,\infty)$ and} \\
    G(\mu)\leq G(-\mu)~\text{whenever}~\mu>0,
    \end{array}
    \label{G1}\tag{\ensuremath{G_2}}
\intertext{An immediate consequence of \eqref{G1} and $\eqref{Wsym}$ is that $E(\abs{u})\leq E(u)$ for every $u\in W_0^{1,p}(B_R(0))$. In particular, whenever $u$ is a minimizer, the nonnegative function 
$\abs{u}$ is a minimizer, too. To obtain symmetry of all minimizers, strict monotonicity of $G$
plays a crucial role:}
{}& (Shape') &&\quad \begin{array}{l}
    G~\text{is strictly decreasing on $[0,\infty)$ and} \\
    G(\mu)\leq G(-\mu)~\text{whenever}~\mu>0,
    \end{array}
    \label{G1p}\tag{\ensuremath{G_2'}}
\end{alignat*}
%Furthermore, w.~l.~o.~g., we shift $G$ to satisfy
%\begin{equation*}
%    G(0)=0. \label{G3}\tag{\ensuremath{G_{3}}}
%\end{equation*}
\begin{rem}
If \eqref{G1} is violated, a minimizer need not exist. For instance,
it is well known that the infimum of
$\int_{B_R(0)} \big[(\abs{\nabla u}-1)^2+u^2\big]\,dx$, $u\in W_0^{1,2}$,
is zero and it is not attained. More generally,
if $\tilde{W}(0)>\min \tilde{W}$ and $G(\mu)>G(0)$ for every $\mu\neq 0$, 
then $\inf E=\min \tilde{W}+G(0)$ and it is not attained.
\end{rem}
\begin{rem}
If $G$ does not satisfy \eqref{G1} (or \eqref{G1p}, respectively), but $\hat{G}:\RR \to \RR$, $\mu\mapsto G(-\mu)$ does (for example, if $G$ is strictly increasing on $\RR$),
our results below still hold with obvious changes. Just consider
$\hat{E}(u):=E(-u)=\int_{B_R(0)} [\tilde{W}(\abs{\nabla u})+\hat{G}(u)]\,dx$ instead of $E$.
\end{rem}
In view of \eqref{W1} and \eqref{G2}, it is natural to consider $E$ as a functional on $W_0^{1,p}(\Omega)$.
A first consequence of the conditions given above is the following
\begin{prop}{(Coercivity of $E$)}\label{1propcoerc} 
Assume \eqref{W0}--\eqref{W2}, \eqref{G0} and \eqref{G2}. Then
$E:W_0^{1,p}(B_R(0))\to \RR$ is well defined and coercive in the
sense that
\begin{align}\label{1E0coerc}
    E(u)&\geq \tilde{\nu}\norm{u}_{W^{1,p}}^p-\tilde{C},
\end{align}
for every $u\in W_0^{1,p}(B_R(0))$, where $\tilde{C}$ and $\tilde{\nu}>0$
are constants independent of $u$.
\end{prop}
\begin{proof}
Using the growth conditions, it is not difficult to show that $E$
is well defined. Furthermore, for $u\in W_0^{1,p}(B_R(0))$, by
virtue of \eqref{W1}, \eqref{G2}, \Hoelder's inequality and
\Poincare's inequality we have that
\begin{align*}
    E(u)&\geq \int_{B_R(0)} \left[\nu_1\abs{\nabla u}^{p}-\nu_3\abs{u}^{p-\varrho}-2C\right]\,dx\\
    &\geq \tilde{\nu}_1\norm{u}_{W^{1,p}}^p-\tilde{\nu}_3\norm{u}_{W^{1,p}}^{p-\varrho}-2C,
\end{align*}
where $\tilde{\nu}_1$ and $\tilde{\nu}_3$ are positive constants depending on
$\nu_1$ and $\nu_3$, respectively, as well as on $p$, $\varrho$ and
$\LM_N({B_R(0)})$. Since $p-\varrho<p$, this immediately implies
\eqref{1E0coerc}.
\end{proof}
%%%%%%%%%%%%%%%%%%%%%%%%%%%%%%%%%%%%%%%%%%%%%%%%%%%%%%%%%%%%%%%%%%%%%%%%%%%%%%%
%%%%%%%%%%%%%%%%%%%%%%%%%%%%%%%%%%%%%%%%%%%%%%%%%%%%%%%%%%%%%%%%%%%%%%%%%%%%%%%
%%%%%%%%%%%%%%%%%%%%%%%%%%%%%%%%%%%%%%%%%%%%%%%%%%%%%%%%%%%%%%%%%%%%%%%%%%%%%%%
\section{Properties of minimizers in the convex case\label{1secExSym}}
In the case of convex $W$, the functional $E$ is weakly lower
semicontinuous, and since it is also coercive by
Lemma~\ref{1propcoerc}, $E$ has a minimum by the direct methods in
the calculus of variations (cf.~\cite{Da89B}, e.g.) in $W_0^{1,p}$. 
%Of course even in this case it is not immediately clear that the
%minimizer has radial symmetry if $W$ itself is radially symmetric in
%the sense of \eqref{Wsym}. 
This section provides several auxiliary
results which are employed to show existence and symmetry of
minimizers for nonconvex $W$ in Section~\ref{1secWNCSYM}. For this
purpose, we will apply the assertions below to the relaxed
functional 
\begin{equation}\label{E0c}%\tag{\ensuremath{E^{**}}}
    E^{**}(u):=\int_{B_R(0)} \left[W^{**}(\nabla u)+G(u)\right]\,dx,
\end{equation}
where $W$ is replaced by its convex envelope $W^{**}$. As a
consequence, we actually could assume that $W=W^{**}$ within this
section. However, the arguments used here do not really exploit
convexity of $W$ (although convexity is always sufficient) which
guarantees the existence of a minimizer. Thus we prefer to use a
more general setting, assuming just those properties of $W$ which
are really needed for the proofs.
%and ignoring the fact that in some
%cases the results might be empty in the sense that they provide
%properties for minimizers which do not exist. 
%Note however that one
%does not achieve a true generalization of the case of convex $W$ in
%this way since a global minimizer for nonconvex $W$ also has to
%minimize the convexified functional where $W$ is replaced with its
%convex envelope $W^{**}$, due to the Relaxation Theorem.
As a first step, we discuss the question of radial symmetry of
minimizers, assuming symmetry of $W$. 
%As mentioned above,
%the existence of a minimizer is quite clear (as long as $W$ is
%convex), thus we only have to show that one (or, preferably, every)
%minimizer is radially symmetric. 
For this purpose, we construct
radially symmetric functions in a suitable way from a given,
possibly asymmetric minimizer. The following lemma provides
sufficient regularity of those functions.
\begin{lem}\label{1lemsymreg}
Let $u$ be (a fixed representative of an equivalence class) in $W^{1,p}(B_R(0))$ 
with a $p\in [1,\infty)$. Then,
for almost every direction $\theta\in S^{N-1}$, the radially symmetric function
\begin{align} \label{1urearrange}
    u_\theta:B_R(0)\to \RR,~u_\theta(x):=u(\abs{x}\theta)
\end{align}
(respectively, its equivalence class) is an element of $W^{1,p}(B_R(0))$. 
If $u\in W_0^{1,p}(B_R(0))$,
then we also have $u_\theta\in W_0^{1,p}(B_R(0))$ for
a.~e.~$\theta\in S^{N-1}$. In any case,
\begin{align} \label{1lsymreg00}
    \nabla u_\theta(x)=\left(\theta\cdot \nabla u(\abs{x}\theta)\right)\frac{x}{\abs{x}},
\end{align}
and in particular,
\begin{align} \label{1lsymreg01}
    \abs{\nabla u_\theta(x)}\leq \abs{\nabla u(\abs{x}\theta)},
\end{align}
for almost every $x\in B_R(0)$ and $\theta\in S^{N-1}$.
\end{lem}
\begin{proof}
We will only give the proof for $u\in W_0^{1,p}(B_R(0))$, 
the modifications for $u\in W^{1,p}(B_R(0))$
are obvious. \\%(note that in this case approximation with smooth
%functions on $\overline{B_R(0)}$ is possible, too, since the
%boundary of the ball is regular enough).\\
Since $u$ is an element of $W_0^{1,p}(B_R(0))$, it can be
approximated with a sequence $u^{(k)}\in C_0^\infty(B_R(0))$, $k\in
\NN$, such that $u^{(k)}\to u$ in $W^{1,p}$. Obviously the radially symmetric functions
$u^{(k)}_\theta$ (obtained from the profiles of $u^{(k)}$ analogously to \eqref{1urearrange})
are elements of $C^\infty(B_R(0)\setminus\{0\})\cap
C(B_R(0))$ and vanish in a vicinity of $\partial B_R(0)$, for every
$k\in\NN$ and every direction $\theta\in S^{N-1}$. Since $\nabla u^{(k)}(0)$ is finite
for fixed $k$, we also have
$u^{(k)}_\theta\in W_0^{1,p}(B_R(0))$. The assertion now follows 
once we show that $u^{(k)}_\theta\to
u_\theta$ in $L^p$ and that $\nabla
u^{(k)}_\theta\to \nabla u_\theta$ in $L^p$ for almost every
$\theta\in S^{N-1}$, where $\nabla u_\theta$ is given by
\eqref{1lsymreg00}. This can be observed in the following way: By
introducing radial coordinates, we have
\begin{align*}
    & \int_{S^{N-1}} \left(\int_{B_R(0)}
    \abs{\nabla u^{(k)}_\theta(x)-\left(\theta\cdot \nabla
    u(\abs{x}\theta)\right)\frac{x}{\abs{x}}}^p\,dx \right)\, d\theta\\
    & = \int_{S^{N-1}} \int_{S^{N-1}}\int_0^R
    \abs{\nabla u^{(k)}_\theta(r\psi)-\left(\theta\cdot \nabla
    u(r\theta)\right)\psi}^pr^{N-1}\,dr\,d\psi\,d\theta\\
    & = \int_{S^{N-1}}\int_{S^{N-1}}\int_0^R
    \abs{\left(\theta\cdot\nabla u^{(k)}(r\theta)\right)\psi-\left(\theta\cdot \nabla
    u(r\theta)\right)\psi}^p\,r^{N-1}dr\,d\psi\,d\theta \\
    %& \hspace{8ex}\text{since $u^{(k)}_\theta$ is radially symmetric} \\
    & \leq \int_{S^{N-1}}\int_{S^{N-1}}\int_0^R
    \abs{\nabla u^{(k)}(r\theta)-\nabla u(r\theta)}^p\,r^{N-1}dr\,d\psi\,d\theta \\
    & = \HMS \int_{B_R(0)} \abs{\nabla u^{(k)}(x)-\nabla u(x)}^p\,dx,
\end{align*}
due to Fubini's Theorem. Since $\nabla u^{(k)}$ converges to $\nabla u$ in $L^p(B_R(0))$, 
this entails that (up to a subsequence)
$\nabla u^{(k)}_\theta\to \nabla u_\theta$ in $L^p(B_R(0))$ as $k\to \infty$, for a.e.~$\theta\in S^{N-1}$.
By a similar calculation we also obtain that $u^{(k)}_\theta\to u_\theta$ in $L^p$ for a.e. $\theta$.
\end{proof}
\begin{rem}
Analogous results about regularity properties of the restrictions of a 
(representative of a) Sobolev function to parallel lines which form a partition
of the domain can be found in \cite{EvGa92B}. However the results
presented there are not directly applicable in the situation of the
lemma above because the lines in radial direction meet at the
origin, thus behaving (mildly) singular.
\end{rem}
As an technical tool in order to prove the symmetry of a whole group
of minimizers (even all for suitable $W$ and $G$), we need the following
elementary characterization of radially symmetric functions:
\begin{lem}\label{1lemradsymonly}
Assume that $u\in W^{1,1}_{\loc}(B_R(0))$ satisfies
\begin{align}\label{1lrso0}
    \nabla u(x)=\lam(x) x\quad\text{for a.~e.~$x\in B_R(0)$},
\end{align}
where $\lam=\lam(x)\in\RR$ is a measurable scalar factor. Then $u$ is radially
symmetric.
\end{lem}
\begin{proof}
Using approximation with smooth functions and Fubini's Theorem as in Lemma~\ref{1lemsymreg}, it
is not difficult to show that the functions $\theta\mapsto
u_r(\theta):=u(r\theta)$, $S^{N-1}\to \RR$, are in
$W^{1,1}(S^{N-1})$ for almost every $r\in (0,R)$.
Furthermore,
\begin{align*}
    D u_r(\theta)h=r Du(r\theta)h~\text{for $h\in T_\theta
    S^{N-1}$}.
\end{align*}
Due to \eqref{1lrso0},
\begin{align*}
    D u_r(\theta)h=r^2\lam(r\theta)(\theta\cdot h)=0,
\end{align*}
since the tangential vector $h\in T_\theta S^{N-1}\subset \RR^N$ is
always orthogonal to $\theta$. Thus $u_r$ is constant on $S^{N-1}$
for almost every $r$. Accordingly, $u$ is constant on the spheres
$\partial B_r(0)$ for almost every $r\in (0,R)$, which entails
radial symmetry.
\end{proof}
With the aid of Lemma~\ref{1lemsymreg} we now can show radial
symmetry of minimizers. %The main tool is the rearrangement 
%\eqref{1urearrange}
%of a given minimizer to a family of radially
%symmetric functions as defined in Lemma~\ref{1lemsymreg}. 
%In contrast
%to this, the method used in \cite{CePe94a} is based on rearranging by averaging on concentric
%spheres. This, too, yields an admissible radially symmetric function. 
%The disadvantage is that its minimizing property can only be shown
%for convex $G$ (using Jensen's inequality), whereas the approach used here does
%not require any such conditions on the shape of $G$. 
\begin{thm} \label{1thmradmincon}
Assume \eqref{W0}, \eqref{W1}, \eqref{Wsym}, \eqref{G0} and \eqref{G2}. Furthermore assume
that $\tilde{W}$ is increasing on $[0,\infty)$ and that $E$ has a global minimizer
$u$ in $W_0^{1,p}$. Moreover, let $M_0\geq 0$ denote a constant such that $\tilde{W}$ is constant on $[-M_0,M_0]$ (note that $M_0=0$ is allowed).
Then we have the following:
\vspace*{-2ex}
\begin{enumerate}
\item At least one global minimizer $u$ of $E$ is radially symmetric. If \eqref{G1} holds, then $u$ can be chosen in such a way that $u\geq 0$ and $\partial_r u\leq -M_0$ almost everywhere.
\item Any minimizer $u$ such that
\begin{align} \label{1trmc0}
    \abs{\nu}>\abs{\partial_r{u}(x)}~\text{implies that}~\tilde{W}(\abs{\nu})>\tilde{W}(\abs{\partial_r{u}(x)}),
\end{align}
for every $\nu\in \RR$ and a.~e.~$x$, is radially symmetric. 
\item Assume in addition that \eqref{G1p} holds. Then every minimizer $u$ of $E$ satisfies
\eqref{1trmc0} and thus is radially symmetric. 
Furthermore, $u$ is either nonnegative or nonpositive in $B_R(0)$.
Here, the latter case can occur only if~\,$G(u)=G(-u)$,
so that $\abs{u}$ is a minimizer, too, then.
If $u$ is nonnegative then we have $\partial_r u\leq -M_0$ almost
everywhere; in particular, $u$ is decreasing in radial direction.
\end{enumerate}
\vspace*{-2 ex}
%Here, $\partial_r u(x):=\nabla u(x)\cdot \frac{x}{\abs{x}}$ is the partial derivative of $u$ in radial
%direction.
\end{thm}
\begin{rem} \label{1remWstrictconvex}
If $\tilde{W}$ is strictly increasing on $[0,\infty)$ (in particular, this is the case if $\tilde{W}$ is strictly convex), \eqref{1trmc0} is automatically satisfied. Hence in that case every minimizer is radially symmetric, even if \eqref{G1p} does not hold.
\end{rem}
\begin{rem}\label{1remWceqWae}
If the monotonicity of $G$ is not strict and $M>0$ (i.e., $0$ is not the unique minimizer of $\tilde{W}$), 
then asymmetric minimizers might exist.
Consider for example the functional
\begin{align*}
	\int_{B_1(0)} \tilde{W}^{**}(\abs{\nabla u})\,dx,
\end{align*}
where $\tilde{W}^{**}(t):=(t^2-1)^2$ for $\abs{t}\geq 1$ and $\tilde{W}^{**}(t):=0$ for $\abs{t}<1$
(which is the convex envelope of $\tilde{W}(t):=(t^2-1)^2$). Obviously, any function $u$ satisfying $\abs{\nabla u}\leq 1$ a.e.~is a minimizer, and it is not difficult to construct one with that property which is not radially symmetric.
One can even construct infinitely many asymmetric functions in $W_0^{1,4}(B_1(0))$ with $\abs{\nabla u}=1$ a.e.,
which also minimize $\int_{B_1(0)} (\abs{\nabla u}^2-1)^2\,dx$.
\end{rem}
\begin{rem}
As we shall see in Theorem~\ref{1thmncradmin} below,
the monotonicity assumption on $W$ can dropped if replaced by
\eqref{W2} and \eqref{G1p} (combined). 
\end{rem}
%%%%%%%%%%%%%%%%%%%%%%%%%%%%%%%%%%%%%%%%%%%%%%%%%%%%%%%%%%%%%%%%%%%%%%%%%%%%%%%%%%%%%%%%%%%
%%%%%%%%%%%%%%%%%%%%%%%%%%%%%%%%%%%%%%%%%%%%%%%%%%%%%%%%%%%%%%%%%%%%%%%%%%%%%%%%%%%%%%%%%%%
%%%%%%%%%%%%%%%%%%%%%%%%%%%%%%%%%%%%%%%%%%%%%%%%%%%%%%%%%%%%%%%%%%%%%%%%%%%%%%%%%%%%%%%%%%%
\begin{proof}[Proof of Theorem~\ref{1thmradmincon}] {\bf(i) Radial symmetry of one minimizer:}\\
In order to show radial symmetry of a minimizer $u$, we first
consider the family $u_\theta\in W_0^{1,p}({B_R(0)})$, $\theta \in S^{N-1}$, of radially symmetric functions defined in Lemma~\ref{1lemsymreg}; in particular, $u_\theta\in W_0^{1,p}({B_R(0)})$ for a.e.~$\theta$.
It satisfies
\begin{align} \label{1trmc1}
    \frac{1}{\HMS}\int_{S^{N-1}} E(u_\theta)\, d\theta \leq  E(u).
\end{align}
This can be observed in the following way: The function $W$ is
radially symmetric by \eqref{Wsym} and increasing in radial
direction, whence by \eqref{1lsymreg01}
\begin{align} \label{1trmc2}
    W(\nabla u_\theta(r\theta))\leq W(\nabla u(r\theta))
\end{align}
for almost every $r\in (0,R)$ and $\theta\in S^{N-1}$. Consequently,
\begin{align*}
    &\int_{S^{N-1}} E(u_\theta) \,d\theta \\
    &\qquad = \int_{S^{N-1}}\int_{S^{N-1}}\int_0^R \left[
    W(\nabla u_\theta(r\psi))+G(u_\theta(r\psi)) \right]\,r^{N-1}dr\,d\psi\,
    d\theta \\
    &\qquad = \int_{S^{N-1}}\int_{S^{N-1}}\int_0^R
    \left[W(\nabla u_\theta(r\theta))+G(u_\theta(r\theta)) \right]\,r^{N-1}dr\,d\psi\,
    d\theta \\
    &\qquad \hspace{8ex}\text{since $u_\theta$ is radially symmetric and $W$ satisfies \eqref{Wsym}} \\
    &\qquad \leq \int_{S^{N-1}}\int_{S^{N-1}}\int_0^R
    \left[W(\nabla u(r\theta))+G(u(r\theta)) \right]\,r^{N-1}dr\,d\psi\,
    d\theta \\
    &\qquad \qquad \qquad \text{due to \eqref{1trmc2}}\\
    &\qquad = \HMS E(u).
\end{align*}
Since $u$ is a minimizer, we know that $E(u)\leq E(u_\theta)$ for
a.~e.~$\theta\in S^{N-1}$. The only way this can coincide with \eqref{1trmc1} is
if
\begin{align}\label{1trmc3}
    E(u)=E(u_\theta),\quad\text{for a.~e.~$\theta\in S^{N-1}$},
\end{align}
i.e., the radially symmetric function $u_\theta$ is a minimizer,
too, for almost every $\theta$. If \eqref{G1} holds, the remaining properties asserted can be achieved by further rearranging $u_\theta$ to another minimizer $v_\theta$ as in step (iii) below. 

{\bf(ii) Radial symmetry of all minimizers satisfying \eqref{1trmc0}:}\\
First observe that as a consequence of the calculation in (i),
\eqref{1trmc3} is possible only if equality holds in \eqref{1trmc2},
for a.~e.~$r$ and $\theta$. By virtue of \eqref{1trmc0} and \eqref{Wsym}, this implies that
\begin{align*}
    \abs{\nabla u_\theta(r\theta)}=\abs{\partial_r u(r\theta)}
    =\abs{\nabla u(r\theta)},\quad\text{for
    a.~e.~$r$, $\theta$}.
\end{align*}
Hence the vector field $\nabla u(x)$ is colinear to $x$ almost
everywhere in $B_R(0)$. Since the only gradient fields on $B_R(0)$
with such a property are gradients of radially symmetric potentials,
as seen in Lemma~\ref{1lemradsymonly}, this proves radial
symmetry of $u$.

{\bf(iii) Common properties of all minimizers, assuming \eqref{G1}:}\\
We define a rearrangement $v_\theta$ of the radially symmetric
minimizers $u_\theta$ by setting
\begin{align*}
    v_\theta'(r):=-\max\left\{\nu\geq 0~\left|~
    \tilde{W}(\nu)=\tilde{W}(\abs{u_\theta'(r)})\right.\right\}~\text{and}
    ~v_\theta(r):=-\int_r^R v_\theta'(s)ds.
\end{align*}
Since $\tilde{W}$ is an even function by \eqref{Wsym},
\begin{align} \label{1propmon1}
    \tilde{W}(v_\theta'(r))=\tilde{W}(u_\theta'(r))~\text{for every}~r\in(0,R).
\end{align}
On the other hand, by \eqref{G1},
\begin{align} \label{1propmon2}
    G(v_\theta(r))\leq G(u_\theta(r))~\text{for every}~r\in(0,R),
\end{align}
because obviously $v_\theta\geq \abs{u_\theta}$. Now
\eqref{1propmon1} and \eqref{1propmon2} imply that
\begin{equation}
\begin{aligned}
       E(v_\theta)&=\HMS\int_0^R \left[\tilde{W}(v_\theta')+G(v_\theta)\right]r^{N-1}dr\\
    &\leq \HMS\int_0^R \left[\tilde{W}(u_\theta')+G(u_\theta)\right]r^{N-1}dr=E(u_\theta).
\end{aligned} \label{1propmon3}
\end{equation}
Recalling that $u_\theta$ is a global minimizer for $E$, we
conclude that equality holds in \eqref{1propmon3} and thus also in
\eqref{1propmon2}, for every $r$, i.e.,
\begin{align} \label{1propmon4}
    G(v_\theta)=G(u_\theta)~\text{on}~(0,R),
\end{align}
Since $v_\theta\geq \abs{u_\theta}$, \eqref{G1p} and \eqref{1propmon4} entail that $v_\theta=\abs{u_\theta}$, and, consequently, $\abs{v_\theta'}=\abs{u_\theta'}$ almost everywhere. 
Since $v_\theta$ is decreasing, this implies that 
$u_\theta'$ cannot change sign on $(0,R)$, and thus
\begin{align}\label{1propmon5}
	\text{either $u_\theta\equiv v_\theta$ or $u_\theta\equiv -v_\theta$.}
\end{align}
Furthermore, by the definition of $v_\theta'$ and the monotonicity of
$\tilde{W}$, we have that
\begin{align*}
    \tilde{W}(v_\theta'(r))<\tilde{W}(\nu)~\text{whenever
    $\abs{\nu}>\abs{v_\theta'(r)}$},
\end{align*}
for a.e.~$r\in (0,R)$ and $\theta\in S^{N-1}$. Thus \eqref{1trmc0} holds (recall that $\abs{v_\theta'(r)}=
\abs{\partial_r v_\theta(r\theta)}=\abs{\partial_r u_\theta(r\theta)}=\abs{\partial_r u(r\theta)}$), and (ii) yields the radial symmetry of $u$. The remaining properties of $u$ claimed in
the theorem now follow directly from \eqref{1propmon5}, \eqref{1propmon4} and the definition 
of the $v_\theta$.
\end{proof}
%%%%%%%%%%%%%%%%%%%%%%%%%%%%%%%%%%%%%%%%%%%%%%%%%%%%%%%%%%%%%%%%%%%%%%%%%%%%%%%%%%%%%%%%%%%
%%%%%%%%%%%%%%%%%%%%%%%%%%%%%%%%%%%%%%%%%%%%%%%%%%%%%%%%%%%%%%%%%%%%%%%%%%%%%%%%%%%%%%%%%%%
%%%%%%%%%%%%%%%%%%%%%%%%%%%%%%%%%%%%%%%%%%%%%%%%%%%%%%%%%%%%%%%%%%%%%%%%%%%%%%%%%%%%%%%%%%%
Concluding this section, we derive a condition for the radial
derivative of a bounded radially symmetric minimizer at the origin,
which can be interpreted as a replacement for the second
Weierstrass--Erdmann corner condition 
at this point. Although it does not contribute to the proof of existence of a minimizer,
it is a qualitative property of radially symmetric 
minimizers which is interesting in its own right. 
Below, we assume that $u$ belongs to $L^\infty(B_R(0))$. Even if $p<N$, this is not a restriction,
since in fact every radially symmetric local minimum $u$ is essentially bounded: 
First note that $u\in C^0[\delta,R]$ for every $\delta>0$ as a consequence of the one--dimensional 
Sobolev imbedding. Thus it is enough to show that $u\in L^\infty_{\loc}(B_R(0))$.
For a proof of the latter see for example \cite{Gia83B} (Theorem 2.1 in Chapter VII).
%The latter holds provided that $G$ satisfies a subcritical growth condition (\eqref{G2} is sufficient), 
%see for example \cite{Gia83B} (Theorem 2.1 in Chapter VII).
%In the radial symmetric case, every critical point belongs to $W^{1,\infty}$ as we show in
%Proposition~\ref{prop0dr} below.
%%%%%%%%%%%%%%%%%%%%%%%%%%%%%%%%%%%%%%%%%%%%%%%%%%%%%%%%%%%%%%%%%%%%%%%%%%%%%%
%%%%%%%%%%%%%%%%%%%%%%%%%%%%%%%%%%%%%%%%%%%%%%%%%%%%%%%%%%%%%%%%%%%%%%%%%%%%%%
\begin{prop} \label{1prop0M}
Assume that \eqref{W0}, \eqref{W1}, \eqref{Wsym} and \eqref{G0}--\eqref{G1}
are satisfied, and that $E$ has a radially symmetric minimizer $u\in W_0^{1,p}({B_R(0)})$ such that
$u\in L^\infty({B_R(0)})$ and $\partial_r u\leq -M$ a.e., where $M$ is given by \eqref{WMpoints}.
Furthermore assume that $\tilde{W}$ is increasing on $[0,\infty)$
and that $G$ satisfies
\begin{align} \label{1prop0M00}
    G(\nu)-G(\mu)\leq
    L \abs{\nu-\mu}~\text{for every $\mu,\nu \in [0,\norm{u}_{L^{\infty}}]$ with $\mu\geq \nu$},
\end{align}
where $L$ is a constant which only depends on $\norm{u}_{L^{\infty}}$.
(In particular, \eqref{1prop0M00} holds if $G$ is locally \mbox{Lipschitz}~continuous.)
Then 
\begin{align} \label{1prop0M01}
    \lim_{r\to 0} \partial_r u(r)= -M,
\end{align}
for a suitable representative of
the $L^p$-function $\partial_r u$. 
\end{prop}
\begin{proof}
Fix $\eps>0$. For each $\delta\in (0,R)$ consider the set
\begin{align*}
    I_\eps^\delta:=\left\{r\in (0,\delta)\mid\partial_r u(r)\leq -M-\eps
    \right\}.
\end{align*}
We show that for each $\eps>0$, there is a corresponding $\delta>0$
such that $I_\eps^\delta$ is of zero measure, which entails
\eqref{1prop0M01} (we assumed that $u'\leq -M$ on $(0,R)$). 
For this purpose we define a
radially symmetric function $u_\delta:B_R(0)\to \RR$ such that in
radial coordinates
\begin{align*}
    \partial_r u_\delta(r):=\left\{
    \begin{array}{ll}
        -M \quad & \text{if $r\in I_\eps^\delta$} \\
        \partial_r u(r) & \text{if $r\in (0,R)\setminus I_\eps^\delta$,}
    \end{array}\right.
    \quad\text{and}~u_\delta(r):=-\int_r^R \partial_r
    u_\delta(s)\,ds.
\end{align*}
Observe that $0\leq u_\delta\leq u$ and $u_\delta\in
W_0^{1,p}(B_R(0))$ for each $\delta$. For fixed $\eps$, there exists a
constant $c_\eps>0$ such that
\begin{align} \label{1prop0M2}
    \tilde{W}(M)-\tilde{W}(\xi)\leq -c_\eps
    \abs{-M-\xi},~\text{whenever $\xi\leq -M-\eps$}
\end{align}
since $\tilde{W}$ is coercive by \eqref{W1} and $W(\xi)>W(-M)$
whenever $\xi<-M$. The 
energy difference now can be estimated as follows:
\begin{align*}
    0&\leq (E(u_\delta)-E(u))\HMS^{-1}\\
    &= \int_{I_\eps^\delta} \left[\tilde{W}(u'_\delta)-\tilde{W}(u')\right]r^{N-1}\,dr
    + \int_0^\delta \left[G(u_\delta)-G(u)\right]r^{N-1}\,dr\\
    &\leq -c_\eps \int_{I_\eps^\delta}
    \abs{u'_\delta-u'}r^{N-1}\,dr
    +\int_0^\delta L\abs{u_\delta-u}r^{N-1}\,dr\\
    &\qquad\qquad\text{due to \eqref{1prop0M2} and \eqref{1prop0M00}}\\
    &\leq -c_\eps \int_{I_\eps^\delta}
    \abs{u'_\delta-u'}r^{N-1}\,dr
    +\int_0^\delta L
    \left(\int_r^\delta \abs{u_{\delta}'(s)-u'(s)}s^{N-1}\,ds\right)\,dr\\
    &\leq \left(-c_\eps +\delta L\right)\int_{I_\eps^\delta}
    \abs{u'_\delta-u'}r^{N-1}\,dr.
\end{align*}
Since the first factor converges to $-c_\eps<0$ as $\delta\to 0$,
the whole expression eventually becomes negative unless
\begin{align*}
    0= \int_{I_\eps^\delta} \abs{u'_\delta-u'}r^{N-1}\,dr
    \geq \int_{I_\eps^\delta} \eps r^{N-1}\,dr
\end{align*}
for small $\delta$, which entails that $I_\eps^\delta$ is of measure
zero.
\end{proof}
%%%%%%%%%%%%%%%%%%%%%%%%%%%%%%%%%%%%%%%%%%%%%%%%%%%%%%%%%%%%%%%%%%%%%%%%%%%%%%%
%%%%%%%%%%%%%%%%%%%%%%%%%%%%%%%%%%%%%%%%%%%%%%%%%%%%%%%%%%%%%%%%%%%%%%%%%%%%%%%
%%%%%%%%%%%%%%%%%%%%%%%%%%%%%%%%%%%%%%%%%%%%%%%%%%%%%%%%%%%%%%%%%%%%%%%%%%%%%%%
\section[Existence and symmetry of minimizers for nonconvex Lagrangians]
{Existence and properties of minimizers for nonconvex Lagrangians\label{1secWNCSYM}}
We first recall some consequences of the relaxation theorem.
\begin{prop}\label{1propcrelax}
Assume that \eqref{W0}--\eqref{Wsym},
\eqref{G0} and \eqref{G2} are satisfied.
%If $E$ attains its infimum, 
Then every minimizer $u$ of $E$ 
(not necessarily radially symmetric) also minimizes
the relaxed functional $E^{**}$ defined in \eqref{E0c}, and it satisfies
$\tilde{W}(\abs{\nabla u})=\tilde{W}^{**}(\abs{\nabla u})$ a.~e., where
$\tilde{W}^{**}$ is the convex envelope of $\tilde{W}$ defined in
\eqref{WWc}. 
\end{prop}
\begin{proof}
This is well known. We sketch the details for the case $p<N$:
By the relaxation theorem (see for example \cite{Da89B}, Theorem 2.1 in Chapter 5),
for every $v\in W_0^{1,p}(B_R(0))$ there exists a sequence $v^s\in W_0^{1,p}(B_R(0))$ such that
$\int_{B_R(0)}W(\nabla v^s)\,dx\to \int_{B_R(0)}W^{**}(\nabla v)\,dx$,
$\nabla v^s\rightharpoonup \nabla v$ weakly in $L^p$ and (by compact imbedding, up to a subsequence) 
$v^s\to v$ in $L^{p^*-\rho}$. As a consequence, we have that $E(v^s)\to E^{**}(v)$, since
the \Nemytskii~operator associated to $G$, i.e., $G:L^{p^*-\rho}(B_R(0))\to L^1(B_R(0))$, $v\mapsto G(v)$, 
is continuous by \eqref{G0} and \eqref{G2}.
In particular, the infima of $E$ and $E^{**}$ coincide (recall the trivial inequality $E^{**}\leq E$). 
Furthermore, $E^{**}(u)=E(u)$ for any minimizer $u$ of $E$
(or, equivalently, $W^{**}(\nabla u)=W(\nabla u)$ a.e.), and any minimizer of $E$ also is a minimizer of $E^{**}$. 
\end{proof}
One major benefit of Proposition~\ref{1propcrelax} is that 
minimizers of $E$ (if they exists) inherit the qualitative properties of minimizers of $E^{**}$. 
In particular, we exploit this to obtain symmetry of all minimizers 
of the nonconvex functional in our main result below.
\begin{thm}\label{1thmncradmin}
Assume that \eqref{W0}--\eqref{Wsym},
\eqref{G0} and \eqref{G2} are satisfied. In addition, suppose that $G$ is either
convex, strictly concave, or of class $C^2$. Then we have the following:
\vspace*{-2ex}
\begin{enumerate}
\item Assume that \eqref{G1} holds.
Then $E$ has a global minimizer in $W_0^{1,p}({B_R(0)})$. At least one minimizer $u$ is radially
symmetric, nonnegative and satisfies $\partial_r u\leq -M$ almost everywhere, where $M$ is defined in \eqref{WMpoints}.
\item Assume that \eqref{G1p} holds. 
Then for every minimizer $u$, $\abs{u}$ has the properties listed in (i); 
in particular, every minimizer is radially symmetric. Furthermore, $u$ does not change sign on $B_R(0)$, and the case $u\leq 0$ is possible only if $G(u)\equiv G(-u)$.
\item Assume that $M=0$, i.e., $\tilde{W}(t)>\tilde{W}(0)$ for every $t\neq 0$. Then $E$ has a global minimizer in $W_0^{1,p}({B_R(0)})$ and every minimizer is radially symmetric.
\end{enumerate}
\vspace*{-1ex}
\end{thm}
\begin{rem} \label{1remunique}
If $G$ is convex and strictly monotone, then the minimizer of $E^{**}$ (and thus, using the relaxation theorem,
also of $E$) is unique \cite{CePe94a}.
In the case of nonconvex $G$ one has uniqueness of the minimizer provided that, in addition to \eqref{G2} and \eqref{G1}, $G$ is of class $C^1$, $\mu \mapsto \mu^{-1} G'(\mu)$ is decreasing on $(0,\infty)$
and $\tilde{W}(t)=At^4-Bt^2+C$ for some constants $A,B>0$, $C\in \RR$, see \cite{KroeKie06a} or Section 1.6 of \cite{diss05B}. (The actual conditions on $\tilde{W}$ are more general than that, but still very restrictive.) 
However, note that this result assumes that the class of candidates only consists of radially symmetric functions (having some qualitative properties which all symmetric minimizers have in common), so it cannot be used to show symmetry, if it is not known in advance. 
\end{rem}
\begin{cor}
Under the assumptions of Theorem~\ref{1thmncradmin} (i), we have that
\begin{align*}
	\partial_r u(x)\to -M\quad\text{as}~\abs{x}\to 0
\end{align*}
for every radially symmetric minimizer $u$ of $E$ such that
$\partial_r u\leq -M$ a.e..
%, where $M$ is defined in \eqref{WMpoints}.
\end{cor}
\begin{proof}
The assertion is due to Proposition~\ref{1prop0M} applied to $E^{**}$. Here, note that convex or concave $G$
automatically is locally Lipschitz continuous.
\end{proof}
%\newpage
\begin{proof}[Proof of Theorem~\ref{1thmncradmin}]~\\
{\bf (i) Existence and further properties of one minimizer assuming \eqref{G1}}\\
First we consider the relaxed energy
\begin{align*}
    E^{**}(u):=\int_{B_R(0)} W^{**}(\nabla u)+G(u)\,dx.
\end{align*}
Note that $W^{**}$ is convex, continuous and satisfies the same
coercivity condition \eqref{W1} as $W$. The functional $E^{**}$ has a
minimizer: Any minimizing sequence for $E$ in $W_0^{1,p}({B_R(0)})$
is bounded in this space by the coercivity of $E^{**}$ inherited from
$E$. Thus it converges weakly up to a subsequence, and the weak
limit $u\in W_0^{1,p}({B_R(0)})$ is a minimizer due to the weak lower
semicontinuity of $E^{**}$ (cf.~\cite{Da89B}, e.g.). As a consequence of 
Theorem~\ref{1thmradmincon} (i) applied to $E^{**}$, we can assume that $u$ has all the properties asserted in 
Theorem~\ref{1thmncradmin} (i). We now
have to show that
\begin{align}\label{1WeqWcae}
    \text{$W(\nabla u)=W^{**}(\nabla u)$ almost everywhere,}
\end{align}
because then $E(u)=E^{**}(u)$. Since $u$ is a minimizer of $E^{**}$ and $E^{**}\leq E$, this entails that
$u$ is a minimizer of $E$, too.
For the proof of~\eqref{1WeqWcae} we proceed as follows:
The convex envelope $\tilde{W}^{**}$ is affine on every connected component
of the detachment set $\{t\in \RR\mid \tilde{W}(t)>\tilde{W}^{**}(t)\}$. Note
that the components are open since $\tilde{W}^{**}$ is continuous, and
each one is bounded due to \eqref{W1}. Since $\partial_r u\notin (-M,M)$ a.e.,
which is the the constant part of $\tilde{W}^{**}$, we now consider all connected components
$H$ of the detachment set such that $\tilde{W}^{**}$
is affine but not constant on the interval $H$. 
In particular, $H\subset (-\infty,0)$ or $H\subset (0,\infty)$ due to 
the symmetry and coercivity of $\tilde{W}^{**}$.
There are at most countably many of those components, and thus it suffices to show
that $S:=\{r\in (0,R)\mid
\partial_r u(r)\in H\}$ is of measure zero, for each such $H$.
If $G$ is convex, the set $S$ is of measure zero as shown in
\cite{CePe05} (if $G$ is convex and of class $C^1$, 
Proposition~\ref{1propaffCP} can be used instead). If $G$ is strictly concave or
if $G$ is of class $C^2$, we arrive at the same conclusion by virtue
of Proposition~\ref{1propaff2} below.

{\bf (ii) Common properties of all minimizers assuming \eqref{G1p}}\\
In view of Proposition~\ref{1propcrelax}, Theorem~\ref{1thmradmincon} (iii) applied to $E^{**}$
yields the assertion. 

{\bf (iii) Existence of one and symmetry of all minimizers if $\mbf{M=0}$}\\
As in (i), we obtain a radially symmetric minimizer $u$ of $E^{**}$
with the aid of Theorem~\ref{1thmradmincon} (i) applied to $E^{**}$. 
(Note however that $u$ might change sign on $(0,R)$ this time.)
Since $0$ is the unique minimizer of $\tilde{W}$ and $\tilde{W}$ is coercive,
we have that $\tilde{W}^{**}(t)>\tilde{W}^{**}(0)=\tilde{W}(0)$ for every $t\neq 0$.
Hence the convex function $\tilde{W}^{**}$ is strictly increasing on $[0,\infty)$
and strictly decreasing on $(-\infty,0]$. In particular,
$\tilde{W}^{**}$ cannot be constant on a connected component $H$ of $\{\tilde{W}^{**}<\tilde{W}\}$,
and $0\notin H$ for any such component.
Reasoning as in (i), we get that $u$ also is a minimizer of $E$.
By virtue of Proposition~\ref{1propcrelax} and the monotonicity of 
$\tilde{W}^{**}$, radial symmetry of all minimizers of $E$ 
is a consequence of Theorem~\ref{1thmradmincon} (ii) applied to $E^{**}$.
\end{proof}
We now derive two results which in particular rule out the
possibility that the radial derivative of a radially symmetric local
minimizer stays in an interval where $\tilde W$ is affine but not constant,
thereby providing the missing piece in the proof of
Theorem~\ref{1thmncradmin} above. We need a few measure--theoretic
notions:
\begin{defn}[\Lebesgue~points and points of density]
Let $f:(0,R) \to \RR$ be locally integrable and let $S\subset \RR$ be
\Lebesgue-measurable. We call $s\in (0,R)$ a \emph{\Lebesgue~point} of
$f$ if 
\begin{align*}
    \frac{1}{h}\int_0^h \abs{f(s+t)-f(s)}\,dt\to 0~\text{as $h\to 0$ ($h\in \RR$)}.
\end{align*}
Furthermore, we call $s\in \RR$ a (measure--theoretic) \emph{point of density} of $S$ if
\begin{align*}
    \lim_{\delta\to 0} \frac{\LM_1(S\cap (s-\delta,s+\delta))}{2\delta}=1.
\end{align*}
\end{defn}
\begin{rem}\label{1remLebesguepoints} 
Almost all points of $(0,R)$ are \Lebesgue~points of $f$, for an
arbitrary function $f\in L_{\loc}^1((0,R))$.
Almost all points of a measurable set $S\subset \RR$ are points of density of $S$. 
In particular,
if the set of points of density of $S$ in $S$ is of
measure zero, then so is $S$. Furthermore, each point of density of
$S$ is an accumulation point of other points of density.
For a proof of the first two assertions see for example~\cite{EvGa92B}; 
the latter two are immediate consequences.
\end{rem}
The proposition below is a variant of a result of A.~Cellina and
S.~Perrotta \cite{CePe94a,CePe05}. Here, we assume more regularity for $G$ to obtain a 
stronger conclusion.
\begin{prop}\label{1propaffCP}
Assume that $W$ satisfies
\eqref{W0}--\eqref{Wsym} and that $G$ is of class
$C^1$ and satisfies \eqref{G2}. Furthermore suppose that $\tilde{W}^{**}$ is affine but not constant
on a bounded open interval $H\subset \RR\setminus\{0\}$, i.e.
\begin{align} \label{1paffCPWaff}
    \tilde{W}^{**}(t)=\alpha t+\beta~\text{for every $t\in H$},
\end{align}
where $\alpha\neq 0$ and $\beta\in \RR$ are constants. Let $u\in
W_0^{1,p}(B_R(0))$ be a local extremal of $E^{**}$ which is radially
symmetric. %and satisfies $\partial_r u\leq 0$ on $B_R(0)$. 
Moreover
let $r_0\in S$ be a \Lebesgue~point of $u'$ as well as a point
of density of $S$, where
\begin{align*}
    S:=\{r\in (0,R)\mid u'(r)\in H\}.
\end{align*}
Then we have that
\begin{align*}
    \liminf_{t\to 0}\frac{G'(u(r_0)+t)-G'(u(r_0))}{t} 
    \leq -\alpha\frac{N-1}{u'(r_0)}\cdot\frac{1}{r_0^2} <0.
\end{align*}
%Moreover, the assertion above also holds if ``local extremal'' is
%replaced by ``critical point'' as long as $E$ is
%\Gateaux-differentiable at $u$.
\end{prop}
\begin{proof}
Assume (w.l.o.g.) that $u$ is a local minimizer. 
Our first aim is to derive the strong \EL~equation 
\eqref{1paffel} below, which would be an immediate consequence of the
fundamental lemma of \DuBoisR~if $E$ is differentiable at $u$ and
$u$ is a critical point. We consider radially symmetric test
functions $\varphi\in W^{1,\infty}(B_R(0))$ with
compact support in $B_R(0)\setminus\{0\}$ such that the following holds for all $r\in(0,R)$:
\begin{align*}
    \text{$u'(r)+h \varphi'(r)\in H$ for every $h\in [-1,1]$, wherever
     $\varphi'(r)\neq 0$}.
\end{align*}
In particular, the latter implies that $\varphi'=0$ outside of $S$ (choose $h=0$). 
An example for a test function satisfying these properties is constructed below.
For every such $\varphi$ and every $t\in \RR$ with $\abs{t}$ sufficiently small,
\begin{align*}
    0 & \leq \frac{1}{\HMS}\left[E^{**}(u+t\varphi)-E^{**}(u)\right]\\
    &=\int_0^R \left[\alpha
    t\varphi'+G(u+t\varphi)-G(u)\right]r^{N-1}dr \\
    &=\int_0^R \left[-\frac{N-1}{r}\alpha
    t\varphi+G(u+t\varphi)-G(u)\right]r^{N-1}dr,
\end{align*}
due to \eqref{1paffCPWaff} and integration by parts. Since $G$ is of
class $C^1$, differentiation with respect to $t$ at $t=0$ entails
\begin{align}\label{1paffelweak}
    0= \int_0^R \left[-\frac{N-1}{r}\alpha
    +G'(u)\right]r^{N-1}\varphi\,dr.
\end{align}
Moreover, we infer that
\begin{align} \label{1paffel}
    -\frac{N-1}{r_0}\alpha+G'(u(r_0))=0~\text{whenever $r_0\in (0,R)$ is a point of
    density of $S$}
\end{align}
by constructing a suitable admissible test function to rule out the
alternative: Assume (w.l.o.g.) that $\frac{N-1}{r_0}\alpha+G'(u(r_0))>0$
at a point of density $r_0\in (0,R)$ of $S$. 
Thus, by continuity,
\begin{align} \label{1paffnotel}
    -\frac{N-1}{r}\alpha+G'(u(r))>0,~\text{for every $r$ in a
    vicinity $(a_1,a_2)$ of $r_0$}.
\end{align}
Here, recall that $u$ is continuous on $(0,R]$ 
due to the one--dimensional Sobolev imbedding. 
For arbitrary $b\in (a_1,a_2)$ we define
$\varphi_b(r):=-\int_r^R \varphi_b'(t)\,dt$, where
\begin{align*}
    \varphi_b'(r):=\left\{\begin{array}{rl}
    \frac{1}{2}\dist{\abs{u'(r)}}{\RR\setminus H} \quad & \text{on $(a_1,b)\cap S$},\\
    -\frac{1}{2}\dist{\abs{u'(r)}}{\RR\setminus H} \quad & \text{on $(b,a_2)\cap S$},\\
    0 \quad & \text{elsewhere.}
    \end{array}\right.
\end{align*}
By continuity, there is a point $b_0\in (a_1,a_2)$ such that $\varphi_{b_0}(a_1)=0$. Thus
$\varphi_{b_0}\geq 0$ on $(0,R)$ and $\supp \varphi_{b_0}\subset [a_1,a_2]\subset (0,R)$. 
Hence $\varphi_{b_0}$ is admissible as a test function for \eqref{1paffelweak}, contradicting
\eqref{1paffnotel}. Here, note that $\varphi_{b_0}$ does not 
vanish almost everywhere since $(a_1,a_2)\cap S$ is of positive measure -- recall that
$r_0\in (a_1,a_2)$ is a point of density of $S$. \\
Now fix a point $r_0\in S$ which is both a point of density of $S$ and a
\Lebesgue~point of $u'$. Since points of density are never isolated,
there exists a sequence $h_n\neq 0$, $h_n\to 0$ such that $r_0+h_n$
is a point of density of $S$, too, for every $n$. Subtracting the
equations \eqref{1paffel} at $r_0+h_n$ and $r_0$ and dividing by
$h_n$, we get
\begin{align}\label{1paffel2}
    -(N-1)\frac{\alpha}{h_n}\left(\frac{1}{r_0+h_n}-\frac{1}{r_0}\right)
    +\frac{1}{h_n}\left[G'(u(r_0+h_n))-G'(u(r_0))\right]=0.
\end{align}
for every $n\in \NN$. 
%Furthermore,
%\begin{align*}
%    \frac{G'(u(r_0+h_n))-G'(u(r_0))}{h_n}=\frac{G'(u(r_0+h_n))-G'(u(r_0))}
%    {u(r_0+h_n)-u(r_0)}\frac{u(r_0+h_n)-u(r_0)}{h_n},
%\end{align*}
Furthermore,
\begin{align*}
    0\neq \frac{u(r_0+h_n)-u(r_0)}{h_n}=\frac{1}{h_n}\int_0^{h_n} u'(r_0+t)\,dt=:d_n,
\end{align*}
where $u(r_0+h_n)=u(r_0)$ is impossible since this would contradict
\eqref{1paffel2}.
Thus \eqref{1paffel2} can be rewritten as
\begin{align} \label{1paffel3}
    \frac{G'(u(r_0+h_n))-G'(u(r_0))}{u(r_0+h_n)-u(r_0)}
    =-\alpha \frac{N-1}{d_n}\frac{1}{h_n}\left(-\frac{1}{r_0+h_n}+\frac{1}{r_0}\right).
\end{align}
Since $r_0\in S$ is a \Lebesgue~point of $u'$, 
%and a point of density of $S=\{u'\in H\}$, 
we also have that 
\begin{align*}
    \lim_{n\to\infty}\,d_n=u'(r_0)\in H
\end{align*}
and passing to the limit in \eqref{1paffel3} yields the assertion. Here, note that
$\alpha$ and $u'(r_0)$ have the same sign:
$\alpha>0$ if $H\subset (0,\infty)$ and 
$\alpha<0$ if $H\subset (-\infty,0)$, since $\tilde{W}^{**}$ is even and increasing on $(0,\infty)$.
\end{proof}
If $G$ is of class $C^2$, Proposition~\ref{1propaffCP} implies that
$G''(u(r_0))<0$ whenever $r_0$ is a point of density of $S=\{\partial_r u\in H\}$ (as well as a \Lebesgue~point of $\partial_r u$). 
In particular, $G$ is strictly concave near $u(r_0)$.
%(Here, note that continuity of the second derivative of $G$ is important for this conclusion,
%so it cannot be easily avoided even if Proposition~label{1propaffCP} 
%somehow could be modified to use less regularity.)
But in fact, such a point of density $r_0$ cannot exist if $u$ is a local minimizer:
\begin{prop}\label{1propaff2}
Assume that \eqref{W0}--\eqref{Wsym}, \eqref{G0} and \eqref{G2} are satisfied. 
Furthermore suppose that
$\tilde{W}^{**}$ is affine on a bounded open interval $H\subset
\RR$, i.e.
\begin{align} \label{1paffWaff}
    \tilde{W}^{**}(t)=\alpha t+\beta~\text{for every $t\in H$},
\end{align}
where $\alpha,\beta\in \RR$ are constants. Let $u\in
W_0^{1,p}(B_R(0))$ be a radially symmetric local minimizer of $E$.
Then any point $r_0\in (0,R)$ such that
\begin{align*}
    \text{$G$ is strictly concave in a vicinity of $u(r_0)$}
\end{align*}
%\begin{align} \label{1paffGconc}
%    G(u(r_0)+\mu)+G(u(r_0)+\mu)-2G(u(r_0))<0\quad\text{for every $\mu\in (-2\delta,2\delta)$,}
%\end{align}
%for some $\delta>0$ small enough, 
is not a point of density of $S:=\{r\in (0,R)\mid u'(r)\in H\}$. 
In particular, if $G$ is of class $C^2$, $H\subset \RR\setminus\{0\}$ and $\alpha\neq 0$, then $S$ is of measure zero due to Proposition~\ref{1propaffCP}.
\end{prop}
\begin{proof}
The proof is indirect. Assume that $r_0\in (0,R)$ is a point of
density of $S$. 
%such that $G''(u(r_0))<0$. 
We choose $\delta>0$ and
a vicinity $(a_1,a_2)$ of $r_0$, $0<a_1<a_2<R$, small enough such that
\begin{equation}
\begin{aligned}
    {}&\text{$G$ is strictly concave on $[-2\delta+u(r_0),2\delta+u(r_0)]$ and}\\
    {}&\text{$\abs{u(r)-u(r_0)}\leq\delta$ whenever $r\in [a_1,a_2]$}.
\end{aligned}
%	\text{$\abs{u(r)-u(r_0)}\leq\delta$ whenever $r\in [a_1,a_2]$}.
	\label{1paffGdelta}
\end{equation}
Now define a radially symmetric test function $\varphi\in
W^{1,\infty}(B_R(0))$ such that the support of $\varphi$ is contained in $[a_1,a_2]$,
\begin{align}
    &\text{$u'(r)+h \varphi'(r)\in H$ for every $h\in
    [-1,1]$, wherever $\varphi'(r)\neq 0$},\label{1paffadmtf} \\
    &\text{$\varphi\neq 0$ on a set of positive measure and}
    \label{1paffnttf}\\
    &\norm{\varphi}_{L^{\infty}}<\delta.\label{1pafftfbnd}
\end{align}
Such a test function can be obtained analogously to the definition of $\varphi_{b_0}$ in the 
proof of Proposition~\ref{1propaffCP}:
For arbitrary $b\in (a_1,a_2)$ let
$\varphi_b(s):=-\int_s^R \varphi_b'(t)\,dt$, where
\begin{align*}
    \varphi_b'(r):=\left\{\begin{array}{rl}
    \frac{1}{2}\dist{\abs{u'(r)}}{\RR\setminus H} \quad & \text{on $(a_1,b)\cap S$},\\
    -\frac{1}{2}\dist{\abs{u'(r)}}{\RR\setminus H} \quad & \text{on $(b,a_2)\cap S$},\\
    0 \quad & \text{elsewhere,}
    \end{array}\right.
\end{align*}
and choose $b_0\in (a_1,a_2)$ in such a way that $\varphi_{b_0}(a_1)=0$.
Then the function $\varphi:=\gamma \varphi_{b_0}$
fulfills our requirements, where $\gamma\in (0,1]$ is a suitable scaling factor ensuring
\eqref{1pafftfbnd}. Since $u$ is a local minimizer of $E$, we have
\begin{equation}\label{1paffEineq}
\begin{aligned}
    0 & \leq E^{**}(u+\varphi)+E^{**}(u-\varphi)-2E_0^{**}(u)\\
    & =\HMS \int_{(a_1,a_2)\cap \{\varphi \neq 0\}}
    \left[G(u+\varphi)+G(u-\varphi)-2G(u)\right]r^{N-1}dr,
\end{aligned}
\end{equation}
due to \eqref{1paffadmtf} and \eqref{1paffWaff}, at least as long as $\gamma$ (and thus $\norm{\varphi}_{W^{1,p}}$) is small enough. However, by \eqref{1pafftfbnd} and \eqref{1paffGdelta}, 
%by \eqref{1pafftfbnd}, \eqref{1paffGdelta} and \eqref{1paffGconc}, 
$G$ is
strictly concave on an interval containing all possible values of
its arguments in \eqref{1paffEineq}, and thus
$G(u+\varphi)+G(u-\varphi)-2G(u)<0$ wherever $\varphi\neq 0$, which
contradicts \eqref{1paffEineq} by virtue of \eqref{1paffnttf}.
\end{proof}
%%%%%%%%%%%%%%%%%%%%%%%%%%%%%%%%%%%%%%%%%%%%%%%%%%%%%%%%%%%%%%%%%%%%%%
%%%%%%%%%%%%%%%%%%%%%%%%%%%%%%%%%%%%%%%%%%%%%%%%%%%%%%%%%%%%%%%%%%%%%%
%%%%%%%%%%%%%%%%%%%%%%%%%%%%%%%%%%%%%%%%%%%%%%%%%%%%%%%%%%%%%%%%%%%%%%
%\nocite{diss05B}
\bibliographystyle{plain}
\bibliography{diss-bib}
%%%%%%%%%%%%%%%%%%%%%%%%%%%%%%%%%%%%%%%%%%%%%%%%%%%%%%%%%%%%%%%%%%%
%%%%%%%%%%%%%%%%%%%%%%%%%%%%%%%%%%%%%%%%%%%%%%%%%%%%%%%%%%%%%%%%%%%
%%%%%%%%%%%%%%%%%%%%%%%%%%%%%%%%%%%%%%%%%%%%%%%%%%%%%%%%%%%%%%%%%%%
\end{document}